\pdfoutput=1
\documentclass[10pt]{article}
\usepackage[a4paper, margin=1in]{geometry}
\usepackage{amsmath, amssymb, amsthm}
\usepackage[utf8]{inputenc}
\usepackage[english]{babel}
\usepackage[numbers]{natbib}
\usepackage{url}
\usepackage[usenames]{color}
\usepackage[colorlinks=false]{hyperref}
\usepackage{cleveref}
\usepackage{tabularx}
\usepackage{xcolor}
\usepackage{mathabx}
\usepackage{enumitem}
\theoremstyle{plain}

\newtheorem{proposition}{Proposition}[section]

\newtheorem{corollary}{Corollary}[section]
\newtheorem{lemma}{Lemma}[section]

\newtheorem{theorem}{Theorem}[section]

\crefname{conjecture}{Conjecture}{Conjectures}
\crefname{theorem}{Theorem}{Theorems}
\crefname{corollary}{Corollary}{Corollaries}
\crefname{lemma}{Lemma}{Lemmas}
\crefname{proposition}{Proposition}{Propositions}
\crefname{remark}{Remark}{Remarks}
\crefname{definition}{Definition}{Definitions}
\crefname{notation}{Notation}{Notations}
\crefname{example}{Example}{Examples}
\crefname{problem}{Problem}{Problems}
\crefname{section}{\S}{Sections}
\newcommand{\floor}[1]{\left\lfloor #1 \right\rfloor}
\newcommand{\Z}{\mathbb{Z}}

\newcommand{\HW}{\operatorname{HW}}

\title{Arithmetic closed forms count the Mersenne primes, the Fermat primes and the twin-prime pairs}

\author{Mihai Prunescu \footnote{Research Center for Logic, Optimization and Security (LOS), Faculty of Mathematics and Computer Science, University of Bucharest, Academiei 14, Bucharest (RO-010014), Romania; e-mail: {\tt mihai.prunescu@gmail.com}. Simion Stoilow Institute of Mathematics of the Romanian Academy, Research unit 5, P. O. Box 1-764, Bucharest (RO-014700), Romania; e-mail: {\tt mihai.prunescu@imar.ro}.}}

\begin{document}

\maketitle

\begin{abstract} \noindent 
We construct closed forms that generate with repetitions all Mersenne primes, respectively all Fermat primes, all twin-prime pairs and all Sophie Germain primes. Also, we construct closed forms that count all Mersenne primes between $0$ and $2^{n+2}- 1$, respectively all Fermat primes between $0$ and $6n+5$ and all twin-prime pairs between $0$ and $n$. {\it Every closed form is an arithmetic term, i. e. a fixed finite composition of the following arithmetic operations: addition, subtraction, multiplication, division with remainder and the exponentiation $2^n$.} While for generating these sets with repetitions, only Wilson's Theorem is applied, for the counting forms we use more specific tests, i.e. Lucas-Lehmer, respectively Pépin, and we apply to some extent Jones' work (see Acta Arithmetica XXXV, pg. 210 - 221, 1979). To count twin primes we apply Clement's Theorem, which is closely  related to Wilson's. A closed form to count the Sophie Germain primes can be constructed similarly.
 \\[2 mm] 
  \end{abstract}

\section{Introduction}\label{Section1}

A {\bf Mersenne} prime is a prime number of the form:
$$ 2^n - 1$$
where $n$ is some natural number. It is known that, for natural numbers $a$ and $p$, if $a^p - 1$ is prime, then $a=2$ or $p=1$. It is also known that if $2^p - 1$ is prime, then $p$ is prime. So the Mersenne primes are sometimes defined as primes of the form $2^p - 1$, where $p$ is prime.   As of 2025, there are $52$ many known Mersenne primes. Also, the largest known $7$ prime numbers are Mersenne primes. The Mersenne primes are in bijection with the even perfect numbers, i.e. even numbers which are equal with the sum of their proper divisors. It is conjectured that there are infinitely many Mersenne primes. See 
\href{https://oeis.org/A000668}{\texttt{OEIS A000668}}, \cite{A000668}.

A {\bf Fermat} prime is a prime number of the form:
$$ 2^n + 1$$
where $n$ is some natural numbers. It is known that if $2^n + 1$ is prime, then $n$ must be zero or a power of two. The numbers of the form $F_k = 2^{2^k}+1$ are called Fermat numbers. Fermat believed that all Fermat numbers are prime numbers, but $F_5$ proved already not to be prime, as observed by Euler. In fact, there are not known Fermat primes for $n > 4$. The Fermat primes are connected with the constructible regular polygons by the Gauss-Wantzel Theorem. The Theorem states that an $n$-sided regular polygon can be constructed with compass and straightedge if and only if $n$ is either a power of $2$ or the product of a power of $2$ and distinct Fermat primes. This is equivalent with the condition that for $n$, Euler's Totient function $\varphi(n)$ is a power of $2$. It is widely believed that $F_4$ is the largest Fermat prime. See \href{https://oeis.org/A019434}{\texttt{OEIS A019434}}, \cite{A019434}

In an article published in 1979 in  Acta Arithmetica  \cite{jones}, James P. Jones reconsidered the sets of all Mersenne primes and, respectively, of all Fermat primes, and showed that these sets are Diophantine. Jones applied methods developed by Yuri Matyiasevich \cite{matiyasevich1993hilbert}, Julia Robinson and by himself, and constructed relatively small Diophantine definitions for these sets. In conclusion, Jones showed that the questions about the finitude of the sets of Mersenne, respectively Fermat numbers, are questions about Diophantine sets. However, although {\it small}, these Diophantine definitions are too complicated to be used to solve the finitude question - or at least, they were not use so far to this end. 

In \cite{prunescushuniaprimes}, a closed form for the prime counting function $\pi(n)$ is given by Prunescu and Shunia.

In the present paper, we show that there exist arithmetic terms able to count the number of all Mersenne primes, respectively the number of all Fermat primes, in a given interval. Again, these expressions are too complicated to be used to solve the finitude problem. One cannot say whether the functions represented by these expressions are bounded or not. {\em However, we show here that the counting functions for Mersenne and Fermat primes can be computed by performing a fixed number of arithmetic operations from the following list:
$$x + y, \,\,\,\,x \dot{-}y, \,\,\,\, xy, \,\,\,\, \lfloor x/y\rfloor, \,\,\,\,2^x, $$
and that the corresponding expressions contain only the variable $n$.} 

Here $x \dot{-} y := \max \{x-y, 0\}$ represents the so called modified subtraction. These operations build terms which are defined on the set of natural numbers and take all their values in this set. As 
$$a \bmod b := a \dotdiv  b \cdot \left \lfloor \frac{a}{b} \right \rfloor ,$$
the operation $a \bmod b$ is an arithmetic term and will be intensively used. Following the customs of  mathematical logic, {\bf arithmetic terms} are defined inductively. All constants $c \in \mathbb N$ and all variables $n_i$ are arithmetic terms. If $t_1$ and $t_2$ are arithmetic terms, then the expressions $(t_1 + t_2)$, $(t_1 \divdot t_2)$, $(t_1 \cdot t_2)$, $\lfloor t_1/t_2\rfloor$ and $2^{t_1}$ are arithmetic terms. All closed forms constructed in this article are arithmetic terms. We will use the notions {\it closed form } and {\it arithmetic term} synonymously, or even the notion {\it arithmetic closed form}. 

We say that a relation $R \subseteq \mathbb N^m$ has a {\bf singlefold} (exponential) Diophantine definition if there is an (exponential) Diophantine expression $E(\vec n, \vec x)$ such that the following conditions are fulfilled:
\begin{enumerate}
    \item For all $\vec n \in \mathbb N^m$ one has:
    $$R(\vec n) \longleftrightarrow \exists\,\vec x \in \mathbb N^k\,\,\,\, E(\vec n, \vec x) = 0. $$ 
    \item For all $\vec n \in \mathbb N^m$ and for any two tuples $\vec x_1, \vec x_2 \in \mathbb N^k$,
    $$ E(\vec n, \vec x_1) =  E(\vec n, \vec x_2) = 0 \longrightarrow \vec x_1 = \vec x_2.$$
\end{enumerate}

The main ingredient in our arithmetic term-constructions is an effective method found by Mazzanti \cite{mazzanti2002plainbases}. In this method, the input is an exponential Diophantine equation $E(\vec n, \vec x)$ depending on $m$ natural numbers building the tuple $\vec n$. The method constructs an arithmetic term $\theta(\vec n,t)$ that expresses the number of solutions $\vec x = (x_1, x_2, \dots x_k)$ such that $E(\vec n, \vec x) = 0$ and such that $\vec x \in \{0, 1, \dots, t\}^k$. Although Jones constructed Diophantine definitions for the set of all Mersenne primes, respectively for the set of all Fermat primes, we cannot use his equations to produce counting arithmetic terms by Mazzanti's method. This is so,  because these Diophantine definitions are not single-fold, and not even finite-fold. Indeed, both Jones' definitions have the form:
$$n \in M \longleftrightarrow \exists\,x_1, \dots, x_k \in \mathbb N\,\,\,\, E(n, x_1, \dots, x_k) = 0,$$
where $M$ is the set of Mersenne primes, respectively Fermat primes. But in both definitions, if $x \in M$, then there are infinitely many tuples $(x_1, \dots, x_k) \in \mathbb N^k$ which satisfy the equation for the given $x$. The Diophantine equations constructed by Jones are based on the known Diophantine definitions of the general exponential relation $\{(x, y, x^y)\} \subset \mathbb N^3$. No singlefold Diophantine definition for the general exponentiation is known so far. According to Matiyasevich, the existence of a singlefold Diophantine definition for the two-argument exponential function is one of the most important open problems.

As we still do not know any single-fold or at least finite-fold Diophantine definition for this relation, we cannot modify Jones' proof for getting Diophantine equations for which Mazzanti's method applies. 

Instead, we use a part of the fore-work did by Jones in order to construct single-fold exponential Diophantine definitions of these sets. This is done as follows:
\begin{enumerate}
    \item For the Mersenne primes, we apply, following Jones, the Lucas-Lehmer divisibility criterion. This criterion says that $2^n-1$ is prime if and only if it divides $s(n-1)$, where $s(n)$ is given by the recurrence:
    $$s(1) = 4,\,\,\,\,s(n+1) = s(n)^2 - 2.$$ 
    Also following Jones, there is a connection between the sequence $s(n)$ and the solutions of the Pell equation $x^2 - 3y^2 = 1$. Namely, if $(x(n), y(n))$ is the $n$-th pair of solutions of this equation, then $s(n) = 2x(2^{n-1})$. We apply an expression representing $x(n)$ as an arithmetic term, which was found by Prunescu and Sauras-Altuzarra in \cite{prunescusauras2024representationcrecursive}, and we derive an arithmetic term representing $s(n)$. Further, we transform the primality condition $(2^n - 1) \mid s(n-1)$ in a convenient singlefold exponential Diophantine equation. Here the word {\it convenient} means simple in $\vec x$. The notion will be defined in section \cref{Section2}. On this equation, one can apply Mazzanti's method to built an arithmetic term counting the Mersenne primes. 
    \item For the Fermat primes, we apply, following Jones, Pépin's test, which says that $N = 2^n + 1$ is prime if and only if: 
    $$ a^{\frac{N-1}{2}} \equiv -1 \mod N, $$ 
    where $a$ is a natural number such that $N \nmid a$ and such that $(a/N) = -1$, where the latest is Legendre's symbol. The test was put by Jones in the form of a combination of exponential Diophantine divisibility conditions which prove to be suitable for Mazzanti's method. 
\end{enumerate}

The paper is organized as follows. In \cref{SectionSupp} we apply Wilson's Theorem, the fact that $n!$ is represented by an arithmetic closed form, and the definitions of the Mersenne primes, Fermat primes, twin-prime pairs, respectively Sophie Germain primes, in order to construct arithmetic closed forms that generate with repetitions these special sets of prime numbers. In \cref{Section2}, Mazzanti's method is presented. We also add a method which leads to the construction of single-fold Diophantine equations with less monomials, respectively using simpler functions. In \cref{Section3}, we recall a method to express C-recursive sequences by arithmetic terms.  In \cref{Section4} we apply the method of \cref{Section3} to construct an arithmetic term expressing the Lucas-Lehmer test sequence. In \cref{Section5} we use this preparation to produce a single-fold exponential Diophantine representation of the Mersenne primes and to construct the corresponding arithmetic term that counts the Mersenne primes. In \cref{Section6} we present Pépin's test and the way Jones rewrote it, while in \cref{Section7} we construct the exponential Diophantine single-fold representation for Fermat primes and the corresponding counting term. In \cref{Section9} we sketch the closed form to count the twin-prime pairs and we make clear that a closed form to count the Sophie Germain primes has an analogous construction. 

Various properties and applications of Mersenne primes and Fermat primes are extremely well presented in the monographs \cite{krizek1} and \cite{krizek2}. 

I consider this paper to be a natural continuation of Jones's paper \cite{jones} from {\it Acta Arithmetica}.

\section{Generation with repetitions}\label{SectionSupp}

\begin{lemma} Given an integer $ n \geq 1 $ for which $ 2^n + 1 $ is prime, then $ n = 2^k $ for some $ k \geq 0 $ or $n = 0$. \end{lemma} 

\begin{proof} Suppose that $n = 2^k d$ with $d > 1$ odd. The polynomial $X^d + 1$ has the root $x = -1$. 
    Then  $(X+1) \,|\,(X^d + 1)$ as polynomials, so $(2^k + 1) \,|\, (2^n +1)$. Contradiction.
\end{proof} 

From now on, we will consider that a Fermat prime is a prime number of the form $2^{2^k} + 1$, excluding the number $2=2^0 + 1$. 

Wilson's Test of primality says that a number $n \in \mathbb N$ is prime if and only if 
$$ (n-1) ! \equiv -1 \mod n.$$

In \cite{prunescusauras2024factorial}, Prunescu and Sauras-Altuzarra observed that the Kalmar elementary function $n \leadsto n!$ is represented by any arithmetic term. (This representation was first used by Julia Robinson in \cite{juliarobinson}, but the authors of \cite{prunescusauras2024factorial} were not aware of this fact.) Namely,
$$ n! = \left \lfloor \frac{r(n)^n}{\binom{r(n)}{n}} \right \rfloor$$
where $r(n)$ is an arithmetic term satisfying:
$$ \forall \, n \in \mathbb N\,\,\,\,r(n) \geq {(n+1)}^{(n+2)},$$
and the binomial coefficient is also expressed as an arithmetic term:
$$\binom{x+y}{x} = \left \lfloor \frac{(1 + 2^{x+y})^{x+y}}{2^{x(x+y)}} \right \rfloor \bmod 2^{x+y}. $$ 
We can eliminate the general exponentiation $(1 + 2^{x+y})^{x+y}$
by using the rule $$a^b = 2^{(ab + a + 1)b} \bmod (2^{ab + a + 1} - a)$$ given by Marchenkov in \cite{marchenkov2007superposition}. In Prunescu and Shunia \cite{prunescushuniaprimes}, the following closed term for the factorial function was used:
\begin{align*}
n! &= \floor{ (8^{n^2})^n \Big/ \binom{8^{n^2}}{n} } .
\end{align*}
Moreover, the expression used there for the binomial coefficient is a more advanced one. Although apparently more complicated, these expression can be transformed in a shorter singlefold exponential Diophantine equation then the usual expression given here. 

Combining an arithmetic term expressing $n!$ with Wilson's Test, Prunescu and Sauras-Altuzarra produced in \cite{prunescusauras2024factorial} an arithmetic term $t(n)$ whose image $t(\mathbb N)$ is the set of prime numbers:
$$t(n) := 2 + ( 2 \cdot n! ) \bmod (n+1). $$
Indeed, $t(n) = n+1$ if $n+1$ is prime, respectively $t(n) = 2$ if $n+1$ is not a prime number. We define also the term:
$$z(n) := 3 + ( 3 \cdot n! ) \bmod (n+1).$$
We observe that for $n \neq 1$, $z(n) = n+1$ if $n+1$ is a prime number, respectively $z(n) = 3$ if $n+1$ is not prime.

Using the definition  of the Mersenne primes, respectively of the Fermat primes, and observing that $3$ is both a Mersenne and a Fermat prime,  we obtain the following: 

\begin{theorem}
    Let $t(n) = 2 + ( 2 \cdot n! ) \bmod (n+1)$ be the arithmetic term generating all primes, and $z(n) = 3 + ( 3 \cdot n! ) \bmod (n+1)$.
    
    For all $n \in \mathbb N$, the following terms $m_1(n)$ and $m_2(n)$ are Mersenne primes and all Mersenne primes are generated with repetitions by $m_1(n)$, ( respectively by $m_2(n)$):
    $$m_1(n) := z(2^{n+1} - 2), $$
    $$m_2(n) := z(2^{t(n)} - 2) . $$

    For all $n \in \mathbb N$, the following terms $f_1(n)$ and $f_2(n)$ are Fermat primes and all Fermat primes are generated with repetitions by $f_1(n)$, (respectively by $f_2(n)$):
    $$f_1(n) := z( 2^{n + 2} ),$$
    $$f_2(n) := z( 2^{2^n} ).$$
\end{theorem} 

As for writing down arithmetic terms counting the Mersenne, respectively the Fermat primes, one should rewind the primality test and write it as a singlefold exponential Diophantine equation, the general Wilson Test prove to be less practical. So, in the coming sections we will rewind more specific primality tests: the Lucas-Lehmer Test for the Mersenne Primes and the Pépin Test for the Fermat primes. 

No we show a similar  closed form for the twin primes, i. e. pairs of natural numbers $(p, p+2)$, both of them prime, see \href{https://oeis.org/A001359}{\texttt{OEIS A001359}}, \cite{A001359}. We observe that the functions:
$$|a - b| := (a \dotdiv b) + (b \dotdiv a),$$
$$\min(a,b) := \left \lfloor \frac{(a + b) \dotdiv |a - b|}{2} \right \rfloor, $$
are arithmetic terms. 

We introduce the following arithmetic terms:
$$p_1(n) := \min(z(n+2), z(n+4)) = 3 + \min( (3 \cdot (n+2)!) \bmod (n+3), (3 \cdot (n+4)!) \bmod (n+5)),$$
$$p_2(n) := p_1(n) + 2.$$

\begin{theorem}
    For all $n \in \mathbb N$, $(p_1(n), p_2(n))$ is a pair of twin primes. Moreover, all twin-prime pairs are generated with repetitions by $(p_1(n), p_2(n))$. 
\end{theorem} 

\begin{proof}
    If $(n+3, n+5)$ is a pair of twin primes, then 
    $$p_1(n) = \min(z(n+2), z(n+4)) = \min(n+3, n+5) = n+3\,\,; \,\,\,\,p_2(n) = n+5.$$
    If not, then at least one of the values $z(n+2)$ and $z(n+4)$ will be equal $3$, while the other value will be $\geq 3$. Then, in this case, $p_1(n) = 3$ and $p_2(n) = 5$. 
\end{proof} 

With this idea, we can easily develop closed forms generating further special primes. We recall that a Sophie Germain prime is a prime $p$ such that $2p+1$ is a prime number, (see \href{https://oeis.org/A005384}{\texttt{OEIS A005384}}, \cite{A005384}). We observe that $p=2$ is a Sophie Germain prime, followed by $p = 3$. We introduce the expression:
$$g(n) := \min(t(n), t(2n + 2)) = 2 + \min ( (2 \cdot n!) \bmod (n+1), ( 2\cdot (2n+2)! ) \bmod (2n+3)).$$

\begin{theorem}
     For all $n \in \mathbb N$, $g(n)$ is a Sophie Germain prime and all Sophie Germain primes are generated with repetitions by $g(n)$. 
\end{theorem} 

\begin{proof}
    If $n+1$ is a Sophie Germain prime, then
    $$g(n) = \min(n+1, 2n+3) = n+1.$$
    If $n+1$ is not a Sophie Germain prime, then at least one of the values $t(n)$ and $t(2n+2)$ is equal $2$, and the other one is $\geq 2$, so in this case $g(n) = 2$. 
\end{proof}


\section{Closed forms for counting solutions}\label{Section2} 

In this section we present Mazzanti's method. Given an exponential Diophantine equation $E(\vec n, \vec x)$ depending on $m$ natural numbers building the tuple $\vec n$. The method constructs an arithmetic term $\theta(\vec n,t)$ that expresses the number of solutions $\vec x = (x_1, x_2, \dots x_k)$ such that $E(\vec n, \vec x) = 0$ and such that $\vec x \in \{0, 1, \dots, t\}^k$.

The following functions are represented by arithmetic terms, see \cite{mazzanti2002plainbases, marchenkov2007superposition, prunescushunia2024gcd}:
\begin{align*}
\binom{a}{b} &= \floor{\frac{(2^a+1)^a}{2^{ab}}} \bmod 2^a , \\
\gcd(a,b) &= \left ( \floor{\frac{5^{ ab(ab + a + b)}}{(5^{a^2 b} - 1)(5^{ab^2}-1)}} \bmod 5^{ab} \right ) - 1, \\
\nu_2(n) &= \floor{\frac{\gcd(n, 2^n)^{n+1} \bmod (2^{n+1}-1)^2}{2^{n+1}-1}} , \\
\HW(n) &= \nu_2\left(\binom{2n}{n}\right) .
\end{align*}  

Here $\binom{a}{b}$ is the binomial coefficient, $\gcd(a,b)$ is the greatest common divisor, $\nu_2(n)$ is the exponent of $2$ in the prime-number decomposition of $n$. Remark that Mazzanti gave another $\gcd$-formula in \cite{mazzanti2002plainbases}. 
The exponent $5$ given in the present formula can be replaced by $2$, but the the formula does not work anymore in the case $(a,b) = (1,1)$. The exponent base $5$ can be replaced by the exponent basis $2$ also using Marchenkov's formula:
$$a^b = 2^{(ab + a + 1)b} \bmod (2^{ab + a + 1} - a),$$
and the same can be done for the general exponentiation in the closed form for the binomial coefficient. 
$HW(n)$ is the number of ones in the binary representation of $n$. The term displayed above is a consequence of Kummer's Theorem, see Matiyasevich \cite{matiyasevich1993hilbert}.

Consider the generalized geometric series:
\begin{align}
S_r (q, t) = \sum_{j=0}^{t-1} j^r q^j .
\end{align} 
Matiyasevich shows that $S_r(q,t)$ is represented by an arithmetic term $G_r(x,y)$ applied to the pair $(q,t)$, see \cite{matiyasevich1993hilbert}, in Appendix.

Let $a,b \in \mathbb{N} $ such that $ 0 \leq a < 2^b$. It is proved in Mazzanti \cite{mazzanti2002plainbases} that the natural number 
\begin{align}
\delta(a,b) := (2^b - 1)(2^b - a + 1) = 2^{2b} - 2^b a + a - 1.
\end{align}
has the property that:
\begin{align*}
\HW(\delta(a,b)) = \begin{cases}
2b, & \text{if } a = 0, \\
b, & \text{if } a \neq 0.
\end{cases}
\end{align*} 
Let $t, w > 0$ be two positive natural numbers. Suppose that the function $f : [0, t-1]^k \rightarrow \mathbb{N}$ satisfies the condition:
$$\forall \,\vec a \in [0, t-1]^k\,\,\,\, 0 \leq f(\vec a) < 2^w.$$
The function  $\beta : \{0, 1, \dots, t-1\}^k   \rightarrow \{0, 1, \dots,  t^k - 1\}$ given by
\begin{align*}
\beta(\vec{a}) = a_1 + a_2 t(\vec{n}) + \cdots + a_k t(\vec{n})^{k-1}.
\end{align*} 
is one to one and onto. 
We define the quantity:
\begin{align*}
M (f,t,w) = \sum_{\vec{a} \in \{0, \dots, t-1\}^k} 2^{2w \beta(\vec{a})} \delta(f(\vec{a}), w).
\end{align*} 
If expressed in base $2$, the number $M(f, t, w)$ is represented by a string of $0$ and $1$ consisting  of the concatenated strings representing $\delta(f(\vec a), w)$ in base $2$. In particular, the number of laticial points:
$$\{\vec a \in [0, t-1]^k : f(\vec a) = 0 \}$$
is exactly:
$$\frac{HW(M(f,t,u))}{u} - t^k.$$
We are interested in functions $f(\vec x)$ which are {\bf simple} exponential polynomials, i.e.  sums of monomials:
\begin{align*}
m(\vec x) = c v_1^{x_1} \cdots v_k^{x_k} x_1^{r_1} \cdots x_k^{r_k},
\end{align*}
where $r_1, \ldots, r_k \geq 0$, $v_1, \ldots, v_k \geq 1$ are integers, and $c \in \Z$. Using the identity:
\begin{align*}
\sum_{\vec{a} \in \{0, \dots, t-1\}^k}
a_1^{r_1} v_1^{a_1}
\cdots a_k^{r_k} v_k^{a_k} = G_{r_1}(v_1, t)
\cdots G_{r_k}(v_k, t)
= \prod_{i=1}^k G_{r_i}(v_i, t) ,
\end{align*}
we compose the arithmetic term:
\begin{align} \label{TermA}
\mathcal{A}_k(m(\vec{x}), t, u)
&= -(2^{u} - 1) \cdot c
\cdot
\prod_{i=1}^k G_{r_i}(2^{2u t^{i-1}} v_i, t)
\end{align}
which depends only on $t$ and $u$. Here $m(\vec x)$, which occurs in the left-hand side, is only a symbol. Also, the free term $c_0$ of $f$ - meaning the exponential monomial $m(\vec x)$ with $v_1 = \cdots = v_k = 1$ and $r_1 = \cdots = r_k = 0$, contributes with the value: 

\begin{align} \label{TermC}
{\mathcal C}_k(c_0, t, u) 
= \frac{(2^{u} - c_0 + 1)(2^{2u t^k} - 1)}{2^{u} + 1}.
\end{align}
The expression $M(f,t,u)$ defined above
is an arithmetic term depending on $t$, $u$, and on all the coefficients of $f(\vec x)$:
$$M(f,t,u) = {\mathcal C}_k(c_0, t, u) + \sum \mathcal{A}_k(m(\vec{x}), t, u). $$ 
Now let $t(\vec n) > 0$ be an arithmetic term and let:
$$f(\vec n, \vec x) = 0$$ 
be an exponential Diophantine equation in parameters $\vec n$, such that:
\begin{enumerate} 
    \item[(i)] The expression $f(\vec n, \vec x)$ is simple in $\vec x$.
    \item[(ii)] The coefficients of $f$ are arithmetic terms in $\vec n$.
    \item[(iii)] The bases of exponentiation $v_i$ are themselves arithmetic terms (or even other kind of functions) in $\vec n$. 
    \item[(iv)] There is an arithmetic term $w(\vec n) > 0$ such that
    $$\forall \,\vec a \in [0, t(\vec n)-1]^k, \quad 0 \leq f(\vec n, \vec a) < 2^{w(\vec n)}.$$
\end{enumerate} 

Altogether, we proved the following:
\begin{theorem}\label{proof:numberofsolutions}
For all $\vec n$,  the number of laticial points:
$$\{\vec a \in [0, t(\vec n)-1]^k: f(\vec n, \vec a) = 0 \}$$
is given by the following arithmetic term in $\vec n$:
$$\frac{HW(M(f, \vec n, t(\vec n),w(\vec n))}{w(\vec n)} - t(\vec n)^k.$$
\end{theorem} 

In this paper we will occasionally use the following observation:
\begin{lemma}
    Let $E(\vec n, \vec x)$ be an exponential polynomials, simple in $\vec x$. Suppose that 
    $$E(\vec n, \vec x) = \sum m_i^+(\vec n, \vec x) - \sum m_j^-(\vec n, \vec x),$$
    where all exponential monomials take non-negative values, and the subtraction sign represents algebraic subtraction, and not truncated subtraction. Then there exists an exponential Diophantine equation $F(\vec n, \vec x, \vec y) = 0$, which has the following properties:
    \begin{enumerate} 
        \item $F(\vec n, \vec x, \vec y)$ is simple in $(\vec x, \vec y)$. 
        \item For all $\vec n \in \mathbb N^m$, $\vec x \in \mathbb N^k$ and $\vec y \in \mathbb N^n$, $F(\vec n, \vec x, \vec y) \geq 0$.
        \item If $E(\vec n, \vec x)$ is an algebraic sum of $m$ monomials, $F(\vec n, \vec x)$ is an algebraic sum of $3m+3$ monomials. 
        \item If $F(\vec n, \vec x, \vec y) = 0$, then $E(\vec n, \vec x) = 0$.
        \item If $E(\vec n, \vec x) = 0$, then there exists a unique $\vec y \in \mathbb N^n$ such that $F(\vec n, \vec x, \vec y) = 0$. 
    \end{enumerate}
\end{lemma}

\begin{proof}
    Of course, if we replace $E(\vec n, \vec x)$ with $F(\vec n, \vec x) = E(\vec n, \vec x)^2$, all conditions are fulfilled, excepting the number of monomials. Instead of this, we introduce for every monomial another new variable $y_i$ by an equation
    $$(y_i - m_i(\vec n, \vec x))^2 = 0,$$
    for monomials with positive sign, respectively
    $$(y_j - m_j(\vec n, \vec x))^2 = 0,$$
    for monomials with negative sign.
    Now we observe that:
    $$ \sum m_i^+(\vec n, \vec x) - \sum m_j^-(\vec n, \vec x) = 0 $$
    is equivalent with:
    $$ \sum y_i - \sum y_j = 0,$$
    respectively with
    $$ ( 2^{\sum y_i} - 2^{\sum y_j} )^2 = 0. $$
    So we define:
    $$F(\vec n, \vec x, \vec y) = \sum (y_i - m_i(\vec n, \vec x))^2 + \sum (y_j - m_j(\vec n, \vec x))^2 + ( 2^{\sum y_i} - 2^{\sum y_j} )^2.  $$
\end{proof} 

One will observe that the terms constructed in the next Sections do not contain only exponentiations with basis $2$, like $2^x$, but also some other constant-base exponentiations like $11^x$ or $12^x$. This is not a problem, as by the identity:
$$x^y = 2^{(xy + x + 1)y} \bmod (2^{xy + x + 1} - x),$$
such exponentiations can be substituted by other arithmetic terms that contain only the function $2^x$.

\section{C-recursive sequences}\label{Section3} 

A \textbf{C-recursive sequence of order $d$} is a sequences $t : \mathbb N \rightarrow \mathbb C$ such that for all $n \geq 0$ one has:
$$
t(n+d) + a_1 t(n+d-1) + \dots + a_{d-1} t(n+1) + a_d t(n) = 0.
$$
The following polynomial is associated with the recurrence rule:
$$
B(x) := 1 + a_1 X + \dots + a_d X^d .
$$

By Theorem 4.1.1 in \cite{stanley}, and Theorem 1 in \cite{PetkovsekZakrajsek}, a sequence is C-recursive if and only if its  generating function:
$$f(z) = \sum_{n \geq 0} t(n) z^n$$
is a rational function $A(z)/B(z)$ with $\deg(A) < \deg(B) = d$. 

The next Lemma has been also given in Everest and al \cite{everestandal}.

\begin{lemma}\label{LemmaGeneralInequality} [Prunescu \& Sauras-Altuzarra, Lemma 4, \cite{prunescusauras2024representationcrecursive}]
If $ s : \mathbb{N} \rightarrow {\mathbb{C}} $ is C-recursive, then there is an integer $ g \geq 1 $ such that $ | t(n) | < g^{n+1} $ for every integer $ n \geq 0 $. \end{lemma}

\begin{theorem}\label{ThmExtraction1} [Prunescu \& Sauras-Altuzarra, Theorem 1 and 2, \cite{prunescusauras2024representationcrecursive}] If $ t : \mathbb N \rightarrow \mathbb{N} $, $f(z)$ is its generating function, $ R $ is the radius of convergence of $ f(z) $ at zero and $ c $, $ m $ and $ n $ are three integers such that $ c \geq 2 $, $ n \geq m \geq 2 $, $ c^{- m} < R $ and $ t ( n ) < c^{n - 2} $ for every integer $ n \geq m $, then $$t(n) = \floor{ c^{n^2} f ( c^{- n} ) } \bmod c^n . $$ 
Also, if $ c \geq 8 $, $c^{-1} < R$ and $ t ( n ) < c^{n / 3} $ for every  $ n \geq 1 $, then the representation works for every $n \geq 1$. 
\end{theorem} 

 In conclusion, if the sequence $s(n)$ is C-recurrent, its generating function $f$ is a rational function $A(x) / B(x)$ with $A(x), B(x) \in \mathbb Z[x]$, $\deg A(x) < \deg B(x) = d$. Moreover, we can choose the polynomials  such that both take positive values for positive $x$ inside the series' disk of convergence around $0$. In conclusion, the sequence $s(n)$ has  {\bf div-mod representation}:
\begin{align*}
s(n) = \floor{ \frac{c^{n^2 + dn} A(c^{-n})}{c^{dn} B(c^{-n})} } \bmod c^n = 
\floor{ \frac{c^{n^2} \tilde A(c^{n})}{ \tilde B(c^n)} } \bmod c^n.    
\end{align*}

\section{Lucas-Lehmer Test }\label{Section4}

The recurrence equation $s(n+1) = s(n)^2-2$ belongs to the doubly exponential sequences studied by Aho and Sloane in \cite{ahosloane}. The closed form presented here is different from the irrational closed forms constructed in \cite{ahosloane} and is based on C-finite sequences, respectively on the fact that this sequence proves to be related to the Pell equation, see \cite{jones}. 

Let $ s $ be the unary operation on the positive integers such that $ s ( 1 ) = 4 $ and $ s ( n + 1 ) = s ( n )^2 - 2 $ for every integer $ n \geq 1 $.

Note that $ s ( n ) = \textrm{A003010} ( n - 1 ) $ for every integer $ n \geq 1 $ (see \href{https://oeis.org/A003010}{\texttt{OEIS A003010}}, \cite{A003010}). In particular, $ s ( 1 ) = 4 $, $ s ( 2 ) = 14 $, $ s ( 3 ) = 194 $, $ s ( 4 ) = 37634 $, $ s ( 5 ) = 1416317954 $, and $ s ( 6 ) = 2005956546822746114 $.

Note also that every term of $ s $ is even.

Let $ x $ be the sequence of integers $ X > 0 $ such that, for some integer $ Y \geq 0 $, the pair $ ( X , Y ) $ is a solution of the Pell's equation $ X^2 - 3 Y^2 = 1 $.

Note that $ x ( n ) = \textrm{A001075} ( n ) $ for every integer $ n \geq 0 $ (see \href{https://oeis.org/A001075}{\texttt{OEIS A001075}}, \cite{A001075}). In particular, $ x ( 0 ) = 1 $, $ x ( 1 ) = 2 $, $ x ( 2 ) = 7 $, $ x ( 3 ) = 26 $, $ x ( 4 ) = 97 $, $ x ( 5 ) = 362 $, $ x ( 6 ) = 1351 $, and $ x( 7 ) = 5042 $.

In Prunescu and Sauras-Altuzarra \cite{prunescusauras2024representationcrecursive} it is shown   that if a Pell equation $x^2 - ky^2 = 1$ (where $k \in \mathbb N$ is not a square) has the trivial solution $(x(0), y(0)) = (1, 0)$ and the fundamental solution $(x(1), y(1))$, then both sequences $x(n)$ and $y(n)$ obey the recurrence rule:
$$w(n+2) - 2x(1)w(n+1) + w(n) = 0,$$
and the generating function of $x(n)$ will be:
$$f(z) = \frac{1 - x(1)z}{1 - 2x(1) z + z^2}.$$
For the Pell equation $x^2 - 3 y^3 = 1$ we find the fundamental solution $(x(1), y(1)) = (2, 1)$. We observe that the expression given by \cref{ThmExtraction1} well represents the sequence $x(n)$ for $n \geq 1$ if $b \geq 11$. We apply \cref{ThmExtraction1} and we obtain: 

\begin{proposition} For all $ n \geq 1 $, we have that $$ x(n) = \floor{\frac{11^{n^2 + 2 n} - 2 \cdot 11^{n^2 + n}}{11^{2 n} - 4 \cdot 11^n + 1}} \bmod{11^n} . $$ \end{proposition}

The following Lemma was given by Jones in \cite{jones}, Lemma 3.2.

\begin{proposition} For all $ n \geq 1 $, we have that $ s ( n ) = 2 x ( 2^{n - 1} ) $. \end{proposition} 

\begin{proof}
    As we observed that all elements of $s(n)$ are even, define the sequence $v(n) = s(n)/2$ and observe that $v(n)$ satisfies the conditions $v(1) = 2$ and $v(n+1)^2 = 2v(n)^2 - 1$. Also, for the Pell equation $x^2 - (a^2 - 1) y^2 = 1$ it is known that $x(2n) = 2x(n)^2 - 1$ and $x(1) = a$. For $a = 2$ one gets the Pell equation $x^2 - 3 y^2 = 1$ and $x(1) = 2$, so $x(2^n) = x(2^{n-1})^2 -1$ and $x(2^{1-1}) = 2$. The sequences $v(n)$ and $x(2^{n-1})$ satisfy the same recurrence with the same initial condition, so therefore are  identical. 
\end{proof} 

\begin{corollary}
    For all $n \geq 1$, we have that
    $$s(n) = 2 \cdot \left ( \left \lfloor \frac{11^{2^{2n-2} + 2^n} - 2 \cdot 11^{2^{2n-2}+2^{n-1}}}{11^{2^n} - 4 \cdot 11^{2^{n-1}} + 1} \right \rfloor \bmod 11^{2^{n-1}}\right ).$$
\end{corollary} 

The following criterion, known as Lucas-Lehmer divisibility criterion, has been used by Jones \cite{jones} and is proved in Lehmer \cite{lehmer}. 

\begin{theorem}\label{lucaslehmer} {\bf (Lucas-Lehmer Test)} Given an integer $ n \geq 3 $, we have that $ 2^n - 1 $ is prime if, and only if, $ 2^n - 1 $ divides $ s ( n - 1 ) $. \end{theorem} 

\section{Counting the Mersenne primes}\label{Section5}

Let $ \mathcal{M} $ be the sequence that maps every integer $ n \geq 0 $ into the cardinality of the set $$ \{ k \in \{ 0 , \ldots , n  \} : \textrm{$ 2^{k + 2} - 1 $ is prime} \} . $$

So for $k \geq 1$ we observe that:
$$2^{k + 2} - 1 \textrm{ is prime } \longleftrightarrow (2^{k + 2} - 1) | s(k+1) \longleftrightarrow (2^{k + 2} - 1) | x(2^k)$$
The last equivalence takes place because $2^{k + 2} - 1$ is odd, and so divides $s(k+1)$ if and only if it divides $x(2^k)$. We observe that this criterion skips the value $2^2 - 1 = 3$, so at the end we should add $1$ to the arithmetic term. 

We define the eleven-variable polynomial
\begin{align*} 
P(n,k,v,a,b,c,d,e,f,x,y) = ( a - 2^k )^2 + \\ 
(f - 2^{2k})^2 +\\
[ 11^{f + 2 a} - 2 \cdot 11^{f + a} - b ( 11^{2 a} - 4 \cdot 11^a + 1 ) - c ]^2 + \\ 
( 11^{2 a} - 4 \cdot 11^a - c - d )^2 + \\ 
( b - 11^a e - x )^2 + \\ 
( 11^a - x - y - 1 )^2 + \\ 
( 2^{k + 2} z - z - x  )^2 + \\ 
( n-k-v )^2. \end{align*}

\begin{lemma}\label{LemmaSetMersenne} Given an integer $ n \geq 0 $, we have that $ \mathcal{M} ( n ) $ is equal $1$ plus the cardinality of the set
$$\{ (k, v,  a , b , c , d , e , f,  x , y ) \in \mathbb N^{10} : P(n,k,v,a,b,c,d,e, f, x,y) = 0 \}.$$  \end{lemma} 

\begin{proof}
    The squared terms must be all $0$. The last one says that $k \leq n$. The first two terms say that $a = 2^k$ and $f = 2^{2k}$. Squared terms three and four say that $b = \lfloor (11^{a^2 + 2a}- 2 \cdot 11^{a^2 + a}) / (11^{2a} - 4 \cdot 11^a + 1) \rfloor$. Squared terms five and six say that $x = b \bmod 11^a$. Altogether $x = x(2^k)$. Squared term seven says that $2^{k+2} - 1$ divides $x(2^k)$. We apply the Lucas-Lehmer Test \ref{lucaslehmer}.
\end{proof} 

We observe that $P$ is already a simple exponential polynomial. But we will apply a two-goal strategy in order to (a) diminish the number of monomials and (b) increase the number of $G_0$-instances in the final term. We introduce following variables:
\begin{eqnarray*}
    g&=& 11^{f + 2a}, \\
    h&=& 11^{f+a}, \\
    i&=& 11^a, \\
    j&=& 11^{2a}, \\
    k&=& 11^{2a} b,\\
    l &=& 11^a b, \\
    m &=& 11^a e,\\
    p &=& 2^{k+2}z.
\end{eqnarray*}
Now the squared terms are rewritten as follows. 
$$[ 11^{f + 2 a} - 2 \cdot 11^{f + a} - b ( 11^{2 a} - 4 \cdot 11^a + 1 ) - c ]^2 = 0,$$
becomes
$$( g - 2 h - k + 4l - b - c )^2 = 0,$$
$$( 2^{g+4l} - 2^{2h +k + b + c})^2 = 0. $$
Further,
$$( 11^{2 a} - 4 \cdot 11^a - c - d )^2=0,$$
becomes
$$ ( j - 4 i - c - d )^2 = 0,$$
$$ (2^j - 2^{4i+c+d})^2 = 0. $$
Also, the expression
$$ ( b - 11^a e - x )^2 = 0, $$
is equivalent with
$$ ( b - m - x )^2 = 0,$$
$$ ( 2^b - 2^{m+x})^2 = 0. $$ 
In the same way,
$$ ( 11^a - x - y - 1 )^2 = 0, $$
$$ ( i - x - y - 1)^2 = 0,$$
$$ (2^i - 2^{x+y+1})^2 = 0,$$ 
and
$$ ( 2^{k + 2} z - z - x  )^2  = 0, $$
$$ ( p - z - x )^2 = 0,  $$
$$ (2^p - 2^{x+z})^2 = 0.  $$
Finally the eighth squared term becomes:
$$ (n-k-v)^2 = 0,$$
$$ (2^n - 2^{k+v})^2 = 0. $$ 
So the equation $P(n,k,v,a,b,c,d,e,f,x,y) = 0$ is equivalent with:
$$( a - 2^k )^2 
+ (f - 2^{2k})^2 
+ (g - 11^{f + 2a})^2 
+ (h - 11^{f+a} )^2 +$$
$$  + ( i - 11^a )^2 +
    (j - 11^{2a})^2 +
    ( k - 11^{2a} b )^2 +
    (l - 11^a b)^2 +$$
$$ + ( m - 11^a e)^2 
   + ( p - 2^{k+2}z)^2 
   + ( 2^{g+4l} - 2^{2h +k + b + c})^2 
   + (2^j - 2^{4i+c+d})^2 + $$ 
$$ +( 2^b - 2^{m+x})^2
+ (2^i - 2^{x+y+1})^2 
+ (2^p - 2^{x+z})^2 
+ (2^n - 2^{k+v})^2.$$
Call the left-hand side of this equation
$$R(n,k,v,a,b,c,d, e, f, g, h, i, j, k, l, m, p, x, y) = R(n, \vec x).$$
There are $16$ squared terms that generate $48$ monomials. The corresponding sum $M(R, n, t, w)$ contains:
\begin{enumerate}
    \item $10$ products of $1$ instance of $G_2$ and $17$ instances of $G_0$,
    \item $10$ products of $1$ instance of $G_1$ and $17$ instances of $G_0$,
    \item $28$ products of $18$ instances of $G_0$ only. 
\end{enumerate}

\begin{lemma}\label{mersenne-t(n)}
    For every $n \in \mathbb N$, the solutions $\vec x \in \mathbb N^{18}$ of the equation
    $$R(n,\vec x) = 0$$
    belong to the cube $[0, t(n) - 1]^{18}$, with
    $$t(n) = 11^{2^{2n + 1}}.$$
\end{lemma} 

\begin{proof}
    Unknowns $k$ and $v$ are $\leq n$. About the variables $z, x, y, e, d, c, b$ we know that they are remainders, complements of remainder or quotients in divisions, where the number to divide is bounded by one of the variables $g, h, i, j, k, l, m, p$. The biggest value is taken by:
    $$g = 11^{f + 2a} \leq 11^{2^{2n} + 2^{n+1}} < 11^{2^{2n+1}}.$$
\end{proof} 

\begin{lemma}\label{mersenne-w(n)}
    For every $n, t \in \mathbb N$ with $n < t$ and for every $\vec x \in  [0, t-1]^{18}$ one has:
    $$0 \leq R(n, \vec x) < 2^{36t+6}.$$
\end{lemma} 

\begin{proof}
    Like in the previous lemma, we will value simplicity over sharpness. The polynomial $P$ is non-negative, being a sum of squares. It is dominated by a polynomial $P'$ obtained from $P$ by replacing all minus signs by plus signs. If we replace all variables with $t$, the third squared term contains the largest monomial, which is $11^{3t}$. So the largest squared term is dominated by:
    $$4 \cdot 11^{6t}.$$
    But there are $16$ squared terms, so their sum is
    $$\leq 16 \cdot 4 \cdot 11^{6t}\leq 64 \cdot 16^{6t} = 2^6 \cdot 2^{36t} = 2^{36t+6}.$$
\end{proof}
So if we replace $t$ by the arithmetic term $t(n)$ which has been previously computed, we can choose
$$w(n) := 36t(n) + 6. $$ 

\begin{theorem}
    The number of exponents $k \leq n$ such that $2^{k+2} - 1$ is prime equals:
    $$\frac{HW(M(R, n, t(n), w(n)))}{w(n)} - t(n)^{18} + 1.$$
\end{theorem} 

\section{Pépin's Test}\label{Section6} 

Traditionally, $2 = 2^0 + 1$ is not considered a Fermat prime, but all the primes of the shape $2^{2^n} + 1$ for $n \geq 0$ are considered Fermat primes. The following criterion, known as Lucas-Lehmer divisibility criterion, has been proved and used by Jones in \cite{jones}, see also Pépin \cite{pepin}. 

\begin{theorem}[\textbf{Pépin's Test}]\label{pepin}  Let $N=2^m+1$ and $a$ such that $N \nmid a$ such that $(a/N)=-1$. Then $N$ is prime if and only if 
    $$a^{\frac{N-1}{2}} \equiv -1 \mod N.$$  \end{theorem} 

Using Pépin's Test, Jones derives in \cite{jones} a test whose implementation avoids towers of exponents. 

\begin{theorem}[\textbf{Jones' Test}]\label{jonestest} The number $6g+5$ is a Fermat prime if and only if the following conditions are fulfilled:
\begin{enumerate}
    \item $(3g+2) \mid 12^{3g+2}$,
    \item $(6g+5) \mid (12^{3g+2} + 1).$
\end{enumerate}
\end{theorem} 

\begin{proof} (by Jones)
    Jones observed that one can choose $a = 12$ in  Pépin's Test as if $N = 2^m + 1$ is prime, then Legendre's symbol $(12/N)=-1$. Also, the Fermat numbers $F_n$ with $n > 0$ satisfy $F_n = 6g+5$, because $2^{2^n} \equiv 4 \mod 6$. Since $N = 6g+5$, $3$ does not divide $N-1$. But $(N-1)/2 = 3g+2 $ divides a power of $12$. So $N-1$ is a power of $2$. Now the condition $(6g+5) \mid (12^{3g+2} + 1)$ is exactly Pépin's Test. 
\end{proof}

\section{Counting the Fermat primes}\label{Section7}

Let $ \mathcal{F} $ be the sequence that maps every integer $ n \geq 0 $ into the cardinality of the set $$ \{ g \in \{ 0 , \ldots , n \} : \textrm{$ 6g+5$ is a Fermat prime} \} . $$

We define the five-variable polynomial
$$Q(n, g, v, a, b) = ( n-g-v )^2 + [12^{3g+2} - a(3g+2)]^2 + [ 12^{3g+2} + 1 - b(6g+5) ]^2. $$

\begin{lemma}\label{LemmaSetFermat} Given an integer $ n \geq 0 $, we have that $ \mathcal{F} ( n ) $ is equal to the cardinality of the set $$ \{ (g, v, a , b ) \in \mathbb N^4 : Q(n, g, v, a, b) = 0 \} . $$ \end{lemma} 

\begin{proof}
We observe that for any such solution $(k, v, a , b ) \in \mathbb N^4 $ one has $g \leq n$ because $(n - g - v)^2 = 0$. The other two squares express Jones' Test conditions.
\end{proof} 

The polynomial $Q(n,g,v,a,b)$ is already simple and can be used to construct an arithmetic term to count the Fermat primes. But we adopt the same strategy as for the Mersenne primes. We define the new variables:
\begin{eqnarray*}
    c &=& 12^{3g+2}, \\
    d &=& ag , \\
    e &=& bg.
\end{eqnarray*} 
We rewrite the three squared terms as follows. The first squared term becomes: 
$$(2^n - 2^{g+v})^2 = 0.$$
Similarly, one has:
$$ [12^{3g+2} - a(3g+2)]^2 = 0,$$
$$( c - 3d - 2a )^2 = 0,$$
$$( 2^c - 2^{3d+2a})^2 = 0, $$
respectively
$$ [ 12^{3g+2} + 1 - b(6g+5) ]^2 = 0, $$
$$ (c + 1 - 6e - 5b)^2 = 0,$$
$$ (2^{c+1} - 2^{6e+5b})^2 = 0. $$
The modified exponential Diophantine equation reads:
$$(c - 12^{3g+2})^2 +
    (d - ag )^2 +
    (e - bg)^2 +$$
$$(2^n - 2^{g+v})^2 +
   ( 2^c - 2^{3d+2a})^2 +
   (2^{c+1} - 2^{6e+5b})^2 = 0. $$
Call the left-hand side of this equation
$$S(n, g, v, a, b, c, d, e) = S(n, \vec x).$$
There are $6$ squared terms that generate $18$ monomials, but as the monomial $2^c$ occurs two times, there will be only $17$. The corresponding sum $M(S, n, t, w)$ contains:
\begin{enumerate} 
    \item $2$ products of $2$ instances of $G_2$ and $5$ instances of $G_0$.
    \item $3$ products of $1$ instance of $G_2$ and $6$ instances of $G_0$.
    \item $2$ products of $3$ instances of $G_1$ and $4$ instances of $G_0$.
    \item $10$ products of $7$ instances of $G_0$ only. 
\end{enumerate}

\begin{lemma}\label{fermat-t(n)}
    For every $n \in \mathbb N$, the solutions $( g, v, a, b, c, d, e)$ of the equation
    $$S(n, g, v, a, b, c, d, e) = 0$$
    belong to the cube $[0, t(n) - 1]^4$, with
    $$t(n) = 12^{3n+3}.$$
\end{lemma} 

\begin{proof}
We see immediately that $g\leq n$, $ v \leq n$, $c \leq 12^{3n+2}$, $d$ and $a \leq c$, and 
$$b \leq \frac{c+1}{6g+5} < \frac{c}{3} < c = 12^{3n+2}.$$
\end{proof}

\begin{lemma}\label{fermat-w(n)}
    For every $n, t \in \mathbb N$ with $n < t$ and for every $( g, v, a, b, c, d, e) \in [0, t-1]^7$ one has:
    $$0 \leq S(n, g, v, a, b, c, d, e) < 2^{22t + 27}.$$
\end{lemma} 

\begin{proof}
     The polynomial $S$ is non-negative, being a sum of squares. It is dominated by a polynomial $S'$ obtained from $S$ by replacing all minus signs by plus signs. If we replace all variables with $t$,  the first and the last  squared terms look bigger. While the last term contains
     $$2^{11t} = 2048^t,$$
     the first term contains
     $$12^{3t+2} < {(12^3)}^{t+1}= 1728^{t+1}.$$
     We observe that $2^{11t+11}$ is bigger than both quantities. It follows that all squared terms are less than
     $$(2 \cdot 2^{11t+11})^2 = 4 \cdot 2^{22t+22},$$
     so their sum is strictly less than
     $$6 \cdot 4 \cdot 2^{22t+22} < 2^{22t + 27}.  $$
\end{proof}
So if we replace $t$ by the arithmetic term $t(n)$ which has been previously computed, we can choose
$$w(n) := 22 t(n) + 27. $$ 

\begin{theorem}
    The number of natural numbers $0 \leq g \leq n$ such that $6g+5$ is a Fermat prime equals:
    $$\frac{HW(M(Q, n, t(n), w(n)))}{w(n)} - t(n)^{7}.$$
\end{theorem} 


\section{Counting the twin-prime pairs}\label{Section9}

In this paragraph we sketch the generation of closed form counting twin-prime pairs, without diving too much into details. 

The following theorem was proved by P. A. Clement in \cite{clement}:

\begin{theorem}
    A necessary and sufficient condition that two integers, $n$ and $n+2$, $n > 1$, both be prime, is that
    $$4[(n-1)! + 1] + n \equiv 0 \mod n(n+2).$$
\end{theorem} 

In order to avoid the exceptional case $n = 1$, we replace $n$ with $k+2$. It follows that $k+2$ and $k+4$ are both prime, if and only if:
$$(k+2)(k+4) \,\,|\, \left (4 \cdot (k+1)! + k + 6 \right ).$$

A natural number $k+2 \leq n$ is the left-hand side of a twin-prime pair, if and only if $k$ satisfies the following system of relations:
\begin{eqnarray*}
    k + 2 & \leq & n \\
    f & = & (k+1)! \\
    (k+2)(k+4) & \mid  & \left (4 f + k + 6 \right )  
\end{eqnarray*} 

\begin{lemma}
    There is a simple singlefold positive exponential Diophantine representation of the factorial function. This is a polynomial exponential expression $E(n, f, \vec x)$, simple in $(f, \vec x)$, such that:
    $$\forall \,(n, f, \vec x) \in \mathbb N^{k+2}\,\,\,\, E(n, f, \vec x) \geq 0,$$
    $$\forall \, n \in \mathbb N\,\,\,\,f = n! \longleftrightarrow \exists\, \vec x \in \mathbb N^k\,\,\,\, E(n, f, \vec x) = 0,$$
    $$\forall \, n, f \in \mathbb N\,\,\forall\, \vec x_1, \vec x_2 \in \mathbb N^k\,\,\,\,E(n, f, \vec x_1) = 0 \wedge E(n, f, \vec x_2) = 0 \rightarrow \vec x_1 = \vec x_2.$$
\end{lemma} 

\begin{proof}
    Such a singlefold positive exponential Diophantine representation is constructed by Prunescu and Shunia in \cite{prunescusauras2024representationcrecursive}. It is just an exercise to rewind the factorial closed form from \cref{SectionSupp} to write down such a representation, following the pattern exposed in \cref{Section5} for the expression $a = x(2^b)$. Namely, one repeatedly applies the rules to express the quotient and the remainder:
    $$a = \lfloor b/c \rfloor \leftrightarrow \exists\,r, v\,\,\,\, b = ac + r \,\wedge \, r + v + 1 = c,$$
    $$a = b \bmod c \leftrightarrow \exists\,q, v\,\,\,\, b = qc + a \,\wedge \, a + v + 1 = c.$$
\end{proof} 

\begin{theorem}
    There is an arithmetic term ${\cal T}(n)$ whose values express the exact number of prime numbers $p$ with $ 0 < p \leq n$ such that $p+2$ is prime.
\end{theorem} 

\begin{proof}
    We consider the exponential Diophantine equation $F(n,k,f, \vec x, a, b) = 0$ defined as follows:
    $$(n - k - 2 - a)^2 + E(k+1, f, \vec x) + (4f + k + 6 - b(k+2)(k+4))^2 = 0. $$ 
    Consider $n$ to be a parameter and all other variables to be unknowns. The exponential Diophantine equation has a solution $(k, f, \vec x, a, b)$ if and only if $k + 2 \leq n$ is a prime number such that $k+4$ is a prime number as well. Moreover, for every such $k$, the tuple $(f, \vec x, a, b)$ is uniquely determined by  $k$. It follows that the number of solutions:
    $$| \{(k,f,\vec x, a, b)\,:\,F(n, k,f,\vec x, a, b) = 0\}| $$
    is equal with ${\cal F}(n)$. It is just matter of patience to find arithmetic terms $t(n)$ and $w(n)$ such that for all $n \in \mathbb N$ one has:
    $$\forall\,(k,f, \vec x, a, b) \in \mathbb N^{k+4} \,\,\,\,F(n,k,f, \vec x, a, b) = 0 \rightarrow (k, f, \vec x, a, b) \in [0, t(n)]^{k+4},$$
    $$\forall\,(k,f, \vec x, a, b) \in [0, t(n)]^{k+4}\,\,\,\,
    0 \leq F(n, k, f, \vec x, a, b) < 2^{w(n)}.$$
    The arithmetic term ${\cal F}(n)$ can be easily obtained by applying Mazzanti's method for $F(n, k, f, \vec x, a, b)$ with unique parameter $n$ and using the arithmetic terms $t(n)$ and $w(n)$. 
\end{proof} 

In a similar way, one can construct an arithmetic term ${\cal G}(n)$ to count the Sophie Germain primes $\leq n$. One can apply directly Wilson's theorem and write the definition of the Sophie Germain primes as:
$$p \,|\,(p-1)! + 1 \,\,\wedge \,\, (2p+1) \,|\, (2p)! + 1,$$
or one can use the more compact criterion inserted by Davide Rotondo in \href{https://oeis.org/A005384}{\texttt{OEIS A005384}}, see \cite{rotondo}, namely:
$$ p = 2 \,\vee \, p(2p+1) \,|\, ((p-1)!)^2 + 6p - 1.$$
Using either definition, one can construct a singlefold non-negative exponential Diophantine definition of the Sophie Germain as for the twin-prime pairs, and transform it in a counting arithmetic closed form. 

\section{Acknowledgement}
    The author thanks Lorenzo Sauras-Altuzarra for many inspiring discussions on this subject.

\end{document}